\documentclass[11pt,a4paper]{article}
\usepackage{amscd}
\usepackage{enumerate}

\usepackage{amsmath, amssymb, latexsym}
\usepackage{bbm}
\usepackage{amsfonts}
\usepackage{amsthm}
\usepackage{relsize}
\usepackage{setspace}
\usepackage{geometry}
\usepackage{url}
\usepackage{xspace}
\usepackage{tocloft}
\usepackage{graphics}
\usepackage{graphicx}
\usepackage{lscape}
\usepackage{microtype}
\usepackage{ulem}

\usepackage[usenames, dvipsnames]{color}
\usepackage[utf8]{inputenc}
\usepackage{tikz}

\usepackage[pagebackref=true]{hyperref}
\usepackage[alphabetic]{amsrefs}
\usepackage[english]{babel}

\usepackage{authblk}

\newtheorem{theorem}{Theorem}[section]
\newtheorem{proposition}[theorem]{Proposition}

\newtheorem{corollary}[theorem]{Corollary}
\newtheorem{lemma}[theorem]{Lemma}

\theoremstyle{definition}
\newtheorem{definition}[theorem]{Definition}

\newtheorem{remark}[theorem]{Remark}

\theoremstyle{problem}

\newcommand{\Aut}{\mathrm{Aut}}
\newcommand{\Opp}{\mathrm{Opp}}
\newcommand{\Ker}{\mathrm{Ker}}
\newcommand{\Imm}{\mathrm{Im}}

\newcommand{\RR}{\mathbb{R}}
\newcommand{\ZZ}{\mathbf{Z}}
\newcommand{\QQ}{\mathbb{Q}}
\newcommand{\CC}{\mathbb{C}}
\newcommand{\NN}{\mathbb{N}}

\newcommand{\cat}{$\mathrm{CAT}(0)$\xspace}
\newcommand{\Min}{\mathrm{Min}}
\newcommand{\Ch}{\mathrm{Ch}}
\newcommand{\St}{\mathrm{St}}
\newcommand{\Stab}{\mathrm{Stab}}
\newcommand{\Fix}{\mathrm{Fix}}

\newcommand{\id}{\operatorname{id}}

\newcommand{\SL}{\operatorname{SL}}

\newcommand{\G}{\operatorname{\mathbb{G}}}
\newcommand{\dist}{\operatorname{dist}}

\newcommand{\st}{\operatorname{St}}

\def\og{\leavevmode\raise.3ex\hbox{$\scriptscriptstyle\langle\!\langle$~}}
\def\fg{\leavevmode\raise.3ex\hbox{~$\!\scriptscriptstyle\,\rangle\!\rangle$}}

\title{Strong transitivity, Moufang's condition\\ and the Howe--Moore property}

\author{Corina Ciobotaru\thanks{corina.ciobotaru@gmail.com}}

\date{October 8, 2021}

\begin{document}
\maketitle

\begin{abstract}
Firstly, we prove that every closed subgroup $H$ of type-preserving automorphisms of a locally finite thick affine building $\Delta$ of dimension $\geq 2$ that acts strongly transitively on $\Delta$ is Moufang. If moreover $\Delta$ is irreducible and $H$ is topologically simple, we show that $H$ is the subgroup $\G(k)^+$ of the $k$-rational points $\G(k)$ of the isotropic simple algebraic group $\G$ over a non-Archimedean local field $k$ associated with $\Delta$. Secondly, we generalise the proof given in \cite{BM00b} for the case of bi-regular trees to any locally finite thick affine building $\Delta$, and obtain that any topologically simple, closed, strongly transitive and type-preserving subgroup of $\Aut(\Delta)$ has the Howe--Moore property.  This proof is different than the strategy used so far in the literature and does not relay on the polar decomposition $KA^+K$, where $K$ is a maximal compact subgroup, and the important fact that $A^+$ is an abelian maximal sub-semi-group. 
\end{abstract}

\section{Introduction}

In this article we will work in the setting where $\Delta$ is a locally finite thick affine building of dimension $\geq 2$ that will always be considered with its complete system of apartments. With respect to that, the ideal boundary (also called the visual boundary) of $\Delta$ is denoted by $\partial_{\infty} \Delta$ and this is endowed with a structure of a spherical building. By \cite[Theorem 11.79]{AB}  the apartments of $\partial_{\infty} \Delta$ are in one to one correspondence with those of $\Delta$.  We denote by $\Aut(\Delta)$ the full group of automorphisms of $\Delta$ and by $\Aut_{0}(\Delta) \leq \Aut(\Delta)$ the subgroup of all type-preserving automorphisms, meaning those automorphisms that preserve a chosen colouring of the vertices of the model chamber and thus of the vertices of the chambers of $\Delta$. 

We say $G \leq \Aut_{0}(\Delta)$ acts \textbf{strongly transitively} on $\Delta$ if for any two pairs $(\mathcal{A}_1,c_1 )$ and $(\mathcal{A}_2 , c_2)$, consisting each of an apartment $\mathcal{A}_i$ of $\Delta$ and a chamber $c_i \in  \Ch(\mathcal{A}_i)$, there exists $g \in G$ such that $g(\mathcal{A}_1) = \mathcal{A}_2$ and $g(c_1) = c_2$.

\begin{definition}
Let $X$ be an affine or spherical building. For a simplex $\sigma \subset X$ let $\st_{X}(\sigma)$ be the set of all simplices in $X$ that contain  $\sigma$ as a sub-simplex, $\st_{X}(\sigma)$ is called the \textbf{star of $\sigma$ in $X$}. A \textbf{root} $\alpha$ of $X$ is a half-apartment of $X$ and its \textbf{boundary wall} (not the ideal boundary) is denoted by $\partial \alpha$.  Let $\mathcal{A}_{X}(\alpha)$ be the set of all apartments of $X$ that contain $\alpha$ as a half-apartment.  \end{definition}

\begin{definition}
 Let $H$ be a closed subgroup of $\Aut_{0}(\Delta)$ and $\alpha$ any root of $\partial_{\infty} \Delta$. The \textbf{root group}  $U_{\alpha}(H)$ associated with $\alpha$ and $H$ is defined to be the set of all elements $g$ of $H$ which fix $\st_{\partial_{\infty} \Delta}(P)$ pointwise  for every panel $P$ in $\alpha - \partial \alpha$. We take $$H^{+}:= \langle U_{\alpha}(H) \; \vert \; \alpha \text{ root of } \partial_{\infty} \Delta  \rangle.$$
\end{definition}

\begin{remark}
Let $H$ be a closed subgroup of $\Aut_{0}(\Delta)$ and $\alpha$ a root of $\partial_{\infty} \Delta$. Then, as the action of $H$ on $\partial_{\infty} \Delta$ is continuous, we have $U_{\alpha}(H)$ is a closed subgroup of $H$. 
\end{remark}

\begin{definition}
\label{def::moufang}
Let $H$ be a closed subgroup of $\Aut_{0}(\Delta)$. We say $H$ is  \textbf{Moufang} if for every root $\alpha$ of $\partial_{\infty} \Delta$ the associated root group $U_{\alpha}(H)$ acts transitively on $\mathcal{A}_{\partial_{\infty} \Delta}(\alpha)$. 
\end{definition}

Recall the following well known results from the literature. For $\G$ a semisimple algebraic group over a non-Archimedean local ﬁeld $k$, of $k$-rank $\geq 2$ (i.e. isotropic over $k$ (see \cite[Definition 20.1]{Bo})), let $G := \G(k)$ be the group of all $k$–rational points of $\G$. By \cite{BT,BTII}, with the group $G$ one associates a locally ﬁnite thick affine building $\Delta$ where $G$ acts by type-preserving automorphisms and strongly transitively. Moreover, by \cite{BoT65}, the group $G$ always possesses a canonical BN-pair, and if moreover the spherical building $\partial_{\infty} \Delta$ at infinity of $\Delta$ is irreducible and of dimension $\geq 2$, then $\partial_{\infty} \Delta$ is Moufang (in the sense of \cite[Definition 7.27]{AB}, see \cite[Theorem 7.59]{AB}). Also, by \cite[Prop.6.14]{BoT}, or \cite[Theorem. 2.3.1]{Ma},  $G^{+} = \langle U_{\alpha}(G) \;  \vert \;  \alpha \text{ is a root of } \partial_{\infty} \Delta \rangle$ is a closed normal subgroup of $G$ and the factor group $G/G^{+}$ is compact.  

If in addition we suppose $\G$ is also (absolutely) almost simple (for the definition see \cite[page 21]{Ma}, thus $\Delta$ is irreducible by \cite[Proposition 14.10(3)]{Bo}), we will call $\G$ an \textbf{isotropic simple algebraic group over the non-Archimedean local ﬁeld $k$}.  By \cite[Remark 6.15]{BoT}, if moreover $\G$ is simply connected (see \cite[Definition 1.4.9]{Ma}), then $G= G^{+}$. 

Let $\G$ be a semisimple, (absolutely) almost simple algebraic group over a local field $k$. Then by \cite[Section 5.8]{Tits74} the spherical building $\partial_{\infty} \Delta$ associated with $G=\G(k)$ (we keep the same notation as above) determines $k$ up to isomorphism, $G^{+}$ up to isomorphism and $\G$ up to special isogeny. If moreover $\G$ is $k$-split then those groups are classified up to isogeny by the type of the root system, and each isogeny class contains a unique (up to isomorphism) simply connected group and a unique adjoint group (see \cite[Section 4]{Tits79}, \cite[Section 2.2]{RR}).

The first result that we prove in this article is the following:
\begin{theorem}
\label{main_thm2}
Let $\Delta$ be a locally finite thick affine building of dimension $\geq 2$. Let $H$ be a closed subgroup of $\Aut_{0}(\Delta)$ that acts strongly transitively on $\Delta$. Then $H$ is Moufang.  
\end{theorem}

To prove Theorem \ref{main_thm2} we follow the same steps and techniques as in \cite[Section 7.G, Theorem D]{CM1} where the case $H= \Aut_{0}(\Delta)$ was treated. 

By \cite[Theorem 1.1 and Theorem 1.2 ]{CM} one can notice that a closed and strongly transitive subgroup of $\Aut_{0}(\Delta)$ is necessarily non-compact. 

Following Weiss \cite[Chapter 28]{Wi} and the remark from \cite[Section 3.B]{CM1} we recall the following non-trivial result: 
\begin{theorem}
\label{main_thm22}
Let  $\Delta$ be an irreducible locally finite thick affine building of dimension $\geq 2$ whose building at infinity $\partial_{\infty} \Delta$ is Moufang (in the sense of \cite[Definition 7.27]{AB}).  Then $\Delta$ is precisely a Bruhat--Tits building associated with the group $\G(k)$  of an isotropic simple algebraic group $\G$  over a non-Archimedean local field $k$, of $k$-rank $\geq 2$.
\end{theorem}


\medskip

The second result of the article, that uses Theorem \ref{main_thm2} above, is the uniqueness of topologically simple, closed and strongly transitive subgroups $H \leq \Aut_{0}(\Delta)$, where $\Delta$ is an irreducible locally finite thick affine buildings of dimension $\geq 2$. Those subgroups are in fact the groups $G^+=\G(k)^{+}$ associated with the isotropic simple algebraic groups $\G$ over non-Archimedean local fields $k$, and so of algebraic origin.

\begin{theorem}
\label{thm::main3}
Let $\Delta$ be an irreducible locally finite thick affine building of dimension $\geq 2$ such that $\Aut_{0}(\Delta)$ acts strongly transitively on $\Delta$. Then the only topologically simple closed and strongly transitive subgroup of $\Aut_{0}(\Delta)$ is the group $G^{+}$ associated with an isotropic simple algebraic group $\G$  over a non-Archimedean local field $k$.
\end{theorem}

The third result of the article concerns the Howe--Moore property \cite{HM79} and we generalize the proof given for the case of bi-regular trees \cite{BM00b} to any locally finite thick affine building $\Delta$. By Theorem \ref{thm::main3}, for an irreducible $\Delta$, we do not obtain new examples of groups having the Howe--Moore property. Still, the proof given in this article is different than the strategy used so far in the literature (see for example \cite{Cio}) and does not relay on the polar decomposition $KA^+K$, where $K$ is a maximal compact subgroup, and the important fact that $A^+$ is an \textit{abelian} maximal sub-semi-group.

The Howe--Moore property \cite{HM79} is a harmonic analytic feature of a locally compact group that has far-reaching applications to many areas of Mathematics, such as homogeneous dynamics, geometry, or number theory. To list some, it has notable consequences regarding discrete subgroups of semi-simple Lie groups, Mostow's rigidity theorem, Ratner's theorems on unipotent flows, or mixing and equidistribution results. 

So far, all known examples of locally compact groups having the Howe--Moore property are:  
\begin{enumerate}
\item
connected, non-compact, simple real Lie groups, with finite center \cite{HM79}, e.g., $\SL(n, \RR)$
\item
isotropic simple algebraic groups over non Archimedean local fields \cite{HM79}, e.g., $\SL(n,\QQ_p)$
\item
 closed, topologically simple subgroups of $\Aut(T)$ with a $2$-transitive action on the boundary of a bi-regular tree $T$, that has valence $\geq 3$ at every vertex, \cite{BM00b}, e.g., the universal group $U(F)^+$ of Burger--Mozes, when F is $2$-transitive.  
 \end{enumerate}
 
 An important fact the above mentioned groups have in common is their polar decomposition $KA^+K$, where $K$ is a maximal compact subgroup and $A^+$ is an abelian maximal sub-semi-group, the latter being crucial in the proof to imply the Howe--Moore property. This was employed in \cite{Cio}  to give a unified proof for all these above  examples.
 
For a historical introduction and basic definitions regarding the Howe--Moore property one can consult \cite[Introduction and Section 2]{Cio}. For completeness, we recall here just some of the very basic notions.

Let $G$ be a locally compact group and $(\pi, \mathcal{H})$ a (strongly continuous) unitary representation  $\pi :G \to \mathcal{U}(\mathcal{H}) $ on a (infinite dimensional) complex Hilbert space $(\mathcal{H}, \left\langle \cdot , \cdot \right\rangle)$. For $v,w \in \mathcal{H}$ the associated \textbf{$(v,w)$-matrix coefficient} is the map $c_{v,w} : G \to \mathbb{C}$ given by $c_{v,w}(g):=\left\langle \pi(g)v ,w \right\rangle$. We say $c_{v,w}$ \textbf{vanishes at infinity} if, for every $\epsilon >0$, the subset $\{g \in G \; \vert \; \vert c_{v,w}(g)\vert \geq \epsilon \}$ is compact in $G$; equivalently,  $\lim\limits_{g \to \infty}  c_{v,w}(g)  =0$, where $\infty$ represents the one-point compactification of the locally compact group $G$.

\begin{definition}
\label{def::HM}
A locally compact group $G$ has \textbf{the Howe--Moore property} if,  any unitary representation of $G$ that is without non-zero $G$--invariant vectors is $C_0$ (i.e. all its matrix coefficients vanish at infinity).
\end{definition}

\begin{theorem}
\label{thm::main}
Let $\Delta$ be a locally finite thick affine building and $G$ a topologically simple, closed, strongly transitive and type-preserving subgroup of $\Aut(\Delta)$. Then $G$ has the Howe--Moore property.
\end{theorem}

\medskip
The two key ideas behind the proof of Theorem \ref{thm::main} are the use of the well known Bruhat  decomposition of a strongly transitive group $G \leq \Aut(\Delta)$ (see Section \ref{sec::main_thm}), and the closed subgroup $G_{\st(\sigma,\partial_{\infty} \mathcal{B})}^{0}$ introduced in Definition \ref{def::st_cone_N}. These allow us to get rid of the so-far-used polar decomposition $KA^+K$ and its \textit{abelian} sub-semi-group $A^+$ techniques.

In the case of a locally finite thick affine building $\Delta$ we expect the Howe--Moore property of a closed, topologically simple and type-preserving $G \leq \Aut(\Delta)$ to be equivalent to the strongly transitive action of $G$ on $\Delta$. Recall, in the case of bi-regular trees, strong transitivity is equivalent to $2$-transitivity. Moreover, it appears that among locally finite thick buildings (so of any Coxeter type), only the affine case provides examples of groups enjoying the Howe--Moore property.  Both of the above mentioned expectations will be studied in a future research project.

\subsection*{Acknowledgements} Firstly, I would like to thank Tobias Hartnick for pushing me to prove Theorems \ref{main_thm2} and \ref{thm::main3} during our research visit at the Institut des Hautes \'{E}tudes Scientifiques (IHES) in Bures-sur-Yvette, Paris, and for very illuminating conversations. As well, I am thankful to Stefan Witzel for pointing out the articles \cite{CM, CM1} and for very useful discussions. The first part of the paper was written down and finalised during my research visit at IHES and at the Institute of Mathematics of the Romanian Academy (IMAR), Bucharest. I would like to thank those two institutions for the perfect working conditions they provide. 

\section{From strong transitivity to Moufang's condition}


Let us recall a fundamental Rigidity Theorem of Tits. To state this, let $X$ be a building and $c$ a chamber of $X$. We denote by $E_{1}(c)$ the set of all chambers of $X$ that are adjacent to $c$. In what follows we denote by $\Delta$ an irreducible locally finite thick affine building of dimension $\geq 2$, if not stated otherwise.

\begin{theorem}(see \cite[Theorem 5.205]{AB})
\label{thm::rigidity}
Let $X$ be a thick spherical building, and let $c$ and $c'$ be opposite chambers in $X$. If an automorphism $\phi$ of $X$ fixes $E_{1}(c) \cup \{c'\}$ pointwise, then $\phi$ is the identity.
\end{theorem}

From Theorem \ref{thm::rigidity} one can deduce simple-transitivity for a root group $U_{\alpha}$ when it enjoys the Moufang property. More precisely:

\begin{corollary}(see \cite[Lemma 7.25(3)]{AB})
\label{cor::simply_tran}
Let $H$ be a closed subgroup of $\Aut_{0}(\Delta)$ and  $\alpha$ any root of $\partial_{\infty} \Delta$. Suppose the root group $U_{\alpha}(H)$  acts transitively on $\mathcal{A}_{\partial_{\infty} \Delta}(\alpha)$. Then the action of $U_{\alpha}(H)$ on $\mathcal{A}_{\partial_{\infty} \Delta}(\alpha)$ is simple-transitive; i.e., for every $u \in U_{\alpha}(H) - \{id\}$ the action of $u$ has no fixed points on $\mathcal{A}_{\partial_{\infty} \Delta}(\alpha)$.
\end{corollary}

\begin{proof}
Suppose there is an element $u \in U_{\alpha}(H) - \{id\}$ that fixes an element of $\mathcal{A}_{\partial_{\infty} \Delta}(\alpha)$, i.e., $u$ fixes pointwise at least one apartment $\Sigma$ of $\partial_{\infty} \Delta$ that contains $\alpha$ as a half-apartment. 
As the affine building $\Delta$ is irreducible, its Coxeter diagram has no isolated nodes, and thus also the one of the spherical building $\partial_{\infty} \Delta$. Then by \cite[Example 3.128]{AB} there is an ideal chamber $c \in \Ch(\alpha)$ having no panel in $\partial \alpha$. Then by the definition of $U_{\alpha}(H)$, the element $u$ fixes $E_{1}(c) \subset \partial_{\infty} \Delta$ and the apartment $\Sigma$, and thus the ideal chamber $c'$ of $\Sigma$ opposite $c$.  By Theorem \ref{thm::rigidity}, as $u$ is also an automorphism of $\partial_{\infty} \Delta$, we must have $u\vert_{\partial_{\infty} \Delta}=\id$, given a contradiction with our assumption. 
\end{proof}

As written in \cite[Remark 7.26]{AB}, the assumption that the Coxeter diagram has no isolated nodes (e.g., the irreducibility condition from Corollary \ref{cor::simply_tran}) is crucial and cannot be removed. 

\begin{corollary}
\label{cor::equal_root_groups}
Let $\alpha$ be any root of $\partial_{\infty} \Delta$ and $H$ a closed subgroup of $\Aut_{0}(\Delta)$. Suppose both root groups $U_{\alpha}(H)$ and $U_{\alpha}(\Aut_{0}(\Delta))$ act transitively on $\mathcal{A}_{\partial_{\infty} \Delta}(\alpha)$. Then $U_{\alpha}(\Aut_{0}(\Delta))= U_{\alpha}(H)$.
\end{corollary}
\begin{proof}
It is obvious that $U_{\alpha}(H) \leq U_{\alpha}(\Aut_{0}(\Delta))$. By Corollary \ref{cor::simply_tran} we have that both groups  $U_{\alpha}(H), U_{\alpha}(\Aut_{0}(\Delta))$ act simple-transitively on $\mathcal{A}_{\partial_{\infty} \Delta}(\alpha)$ and so they must be equal.
\end{proof}

\begin{lemma}
\label{lem::cocompact1}
Let $\Delta$ be a locally finite thick affine building. Let $H$ be a closed subgroup of $\Aut_{0}(\Delta)$ that acts strongly transitively on $\Delta$. Then for every $\xi \in \partial_{\infty} \Delta$ the group $H_{\xi}$ acts cocompactly on $\Delta$ and transitively on the set $\Opp(\xi) \subset \partial_{\infty} \Delta$ of all ideal points opposite $\xi$.
\end{lemma}
\begin{proof}
Fix an apartment $\Sigma$ in $\Delta$. By the strong transitivity of $H$ on $\Delta$, it is enough to consider only ideal points $\xi \in \partial_{\infty} \Sigma$. For what follows, fix such $\xi$. Then there is a unique, minimal-dimensional ideal simplex $\sigma(\xi) \subset \partial_{\infty} \Sigma$ such that $\xi$ is an interior point of $\sigma(\xi)$. Then choose an ideal chamber $c \in \Ch(\Sigma)$ (not necessarily unique), with $\sigma(\xi) \subset c$ and notice $H_{c} \leq H_{\xi}$. It is then enough to prove that for every $c \in \Ch(\Sigma)$, the group $H_{c}$ acts cocompactly on $\Delta$. Indeed, fix $c \in \Ch(\Sigma)$. Since $H$ is strongly transitive on $\Delta$, by \cite[Sec. 17.1]{Ga}, $H$ is strongly transitive on $\partial_{\infty} \Delta$ as well. In particular, we have that $H_{c}$ acts transitively on the set $\Opp(c)$ of all ideal chambers of $\partial_{\infty} \Delta$ opposite $c$. As for every chamber $C \in \Ch(\Delta)$ there is an apartment $\Sigma_1 $ in $\Delta$  containing $C$, and $c$ as an ideal chamber of $\Ch(\Sigma_1)$, we are left with proving that $\Stab_{H_c}(\Sigma):=\{g \in H_{c} \; \vert \; g(\Sigma)=\Sigma\}$ acts cocompactly on $\Sigma$.  To do that, again by the strong transitivity of $H$ on $\Delta$, we know that $\Stab_{H} (\Sigma) \backslash \Fix_{H}(\Sigma)$ is the affine Weyl group $W$ corresponding to the affine building $\Delta$, and that W is generated by the reflexions through the walls of a chamber in $\Ch(\Sigma)$ (see \cite[Section 5.2]{Ga}). We also know that $W$ is simply transitive on special vertices of $\Sigma$ of the same type, and for every two special vertices $v_1,v_2$ of $\Sigma$ of the same type there is a translation automorphism in $\Stab_{H}(\Sigma)$ of the apartment $\Sigma$, sending $v_1$ to $v_2$. In particular, such translation automorphism fixes pointwise the ideal boundary $\partial_{\infty} \Sigma$ of $\Sigma$, and thus belonging to the group $\Stab_{H_c}(\Sigma) \leq H_c$. This remark easily implies $\Stab_{H_c}(\Sigma)$ acts cocompactly on $\Sigma$ and so proving the claim.

Let us now prove the transitivity of $H_{\xi}$ on the set $\Opp(\xi)$. As above, fix an apartment $\Sigma$ of $\Delta$ with $\xi \in \partial_{\infty} \Sigma$, an ideal chamber $c \in \Ch(\partial_{\infty} \Sigma)$ with $\xi \in c$, and let $\xi_{-} $ be the unique ideal point in $\partial_{\infty}\Sigma$ opposite $\xi$. Take $\xi' \in \Opp(\xi)$, with $\xi' \neq \xi_{-}$. Then there is an apartment $\Sigma'$ of $\Delta$ such that $c \subset \Ch(\partial_{\infty}\Sigma')$ and $\xi' \in \partial_{\infty}\Sigma'$. As $H_c \leq H_{\xi}$ acts transitively on $\Opp(c)$, there is $g \in H_c$ with $g(\Sigma') = \Sigma$ and $g(c)=c$. In particular, $g(\xi') \in \partial_{\infty} \Sigma$ and $g(\xi)=\xi$. As $\xi_{-}, \xi'$ are opposite $\xi$, we must have $g(\xi')=\xi_{-}$. The lemma is proven.
\end{proof}

Let $\xi \in \partial_{\infty} \Delta$ and choose $\xi_{-} \in \partial_{\infty} \Delta$ an ideal point opposite $\xi$. Denote by $P(\xi, \xi_{-})$ the set of all apartments of $\Delta$ that contain $\xi$ and $\xi_{-}$ in their ideal boundary at infinity. As $\xi$ is an ideal point, there is a unique minimal-dimensional ideal simplex $\sigma(\xi) \subset \partial_{\infty} \Delta$ such that $\xi$ is an interior point of $\sigma(\xi)$. Let $\sigma_{-}(\xi) \subset \partial_{\infty} \Delta$ be the corresponding ideal simplex, opposite $\sigma(\xi)$, that contains $\xi_{-}$ as an interior point.  Denote by $n_{\xi}$ the dimension (with respect to $\partial_{\infty} \Delta$) of the ideal simplex $\sigma(\xi)$. Then $P(\xi, \xi_{-})= \RR^{n_{\xi}+1} \times \Delta_{\xi}$, where $\Delta_{\xi}$ is a sub-building of $\Delta$, this is proved by Rousseau \cite[4.3]{Rou} in the more general setting of masures.

\begin{lemma}
\label{lem::cocompact2}
Let $\Delta$ be a locally finite thick affine building. Let $H$ be a closed subgroup of $\Aut_{0}(\Delta)$ that acts strongly transitively on $\Delta$. Let $\xi \in \partial_{\infty} \Delta$ and choose $\xi_{-} \in  \partial_{\infty} \Delta$ opposite $\xi$. Then the closed group $H_{\xi,\xi_{-}} \leq H_{\xi}$ acts cocompactly on $P(\xi, \xi_{-})$.
\end{lemma}
\begin{proof}
Fix an apartment $\Sigma$ of $\Delta$ such that $\xi,\xi_{-} \in \partial_{\infty} \Sigma$, and an ideal chamber $c \in \Ch(\partial_{\infty}\Sigma)$ with $\xi \in c$. Then $\Sigma \in P(\xi, \xi_{-})$. Take another apartment $\Sigma' \in P(\xi, \xi_{-})$ with $\Sigma \neq \Sigma'$, and an ideal chamber $c' \in \Ch(\partial_{\infty} \Sigma')$ with $\xi \in c'$. By \cite[Sec. 17.1]{Ga}, $H$ acts strongly transitively on $\partial_{\infty} \Delta$. Then there is $g \in H$ such that $g(c')=c$ and $g(\Sigma')=\Sigma$. Since $\Sigma, \Sigma' \in  P(\xi, \xi_{-})$, we must have $g(\xi)=\xi$ and $g(\xi_{-})=\xi_{-}$. Therefore, $g \in H_{\xi,\xi_{-}}$ and this shows $H_{\xi,\xi_{-}}$ acts transitively on the set of apartments $P(\xi, \xi_{-})$. To prove the lemma it is enough to show that $\Stab_{H_{\xi,\xi_{-}}}(\Sigma)$ acts cocompactly on $\Sigma$. As used in the proof of Lemma \ref{lem::cocompact1}, the affine Weyl group $W$ of $\Delta$, which is isomorphic to $\Stab_{H} (\Sigma) \backslash \Fix_{H}(\Sigma)$, contains translation automorphisms in $\Stab_{H} (\Sigma)$ that sends any special vertex of $\Sigma$ to any other special vertex of $\Sigma$ of the same type. And all those translation automorphisms in $\Stab_{H} (\Sigma)$ do pointwise stabilise the ideal boundary $\partial_{\infty} \Sigma$. Thus, they are all elements of $H_{\xi,\xi_{-}}$, and so $\Stab_{H_{\xi,\xi_{-}}}(\Sigma)$ indeed acts cocompactly on $\Sigma$ as claimed, implying the lemma.  
\end{proof}

Let us now recall some definitions and results from \cite{CM}.

\begin{definition}[see Introduction \cite{CM}]
\label{def::unipotent_radical}
Let $(X,d)$ be a proper \cat space, $H \leq Isom(X)$ a closed subgroup and $\xi$ in the visual boundary $\partial_{\infty}X$ of $X$. Associated with $H_{\xi}$ we define
$$H_{\xi}^{u}:=\{ g \in H \; \vert \; \lim\limits_{t \to \infty} d(g(r(t)), r(t)) = 0, \; \forall \text{ geodesic rays } r: [0,\infty) \to X \text{ with } r(\infty)=\xi \}.$$
Notice, $H_{\xi}^{u}$ is a closed normal subgroup of $H_{\xi}$.
\end{definition}

\begin{remark}
\label{rem::unipotent_radical}
Let $\Delta$ be a locally finite thick affine building and $H$ a closed subgroup of $\Aut_{0}(\Delta)$. Then it is easy to see directly from Definition \ref{def::unipotent_radical} that for every $\xi \in \partial_{\infty} \Delta$ the subgroup $H_{\xi}^{u}$ fixes pointwise any ideal chamber of $\partial_{\infty} \Delta$ that contains $\xi$.
\end{remark}

Below we state \cite[Theorem J]{CM} in the spacial case of $\Delta$ being a locally finite thick affine building.
\begin{theorem}(Levi decomposition I)
\label{thm::levi}
Let $\Delta$ be a locally finite thick affine building and $H$ a closed subgroup of $\Aut_{0}(\Delta)$ that acts strongly transitively on $\Delta$. Let $\xi \in \partial_{\infty} \Delta$. Then for every $\xi' \in \Opp(\xi)$ we have a decomposition 
$$H_{\xi} = H_{\xi, \xi'} \cdot H_{\xi}^{u}$$
which is almost semi-direct in the sense that $H_{\xi, \xi'} \cap H_{\xi}^{u}$ is compact. In particular, $H_{\xi}^{u}$ acts transitively on $\Opp(\xi)$.
\end{theorem}
\begin{proof}
For the complete proof see \cite{CM}. Here we will just give the main idea of the proof. Fix for what follows some $\xi' \in \Opp(\xi)$ and consider the set $P(\xi, \xi')$. By Lemma \ref{lem::cocompact1} we know that $H_{\xi}$ acts transitively on $\Opp(\xi)$, implying that for every $\xi'' \in \Opp(\xi)$, there exists $g \in H_{\xi}$ such that $g(\xi'')=\xi'$ and in particular, $g(P(\xi, \xi''))=P(\xi, \xi')$. Moreover, similarly as for retractions with respect to an ideal chamber in $\partial_{\infty} \Delta$, one can intuitively consider a canonical retraction $\rho_{\xi}$ with respect to the ideal point $\xi$, sending each set $P(\xi, \xi'')$ to $P(\xi, \xi')$, for every $\xi'' \in \Opp(\xi)$. The \textbf{horoaction} $\omega_{\xi}: H_{\xi} \to Isom(P(\xi, \xi'))$ defined in \cite[Definition 3.1]{CM} is aiming to track the action of an element $g \in H_{\xi}$, via the retraction $\rho_{\xi}$ on the space $P(\xi, \xi')$ (i.e., $\rho_{\xi}(g)$). Then it is easy to see that $\Ker(\omega_{\xi})= H_{\xi}^{u}$ and by the transitivity of $H_{\xi}$ on $\Opp(\xi)$, $\Imm(\omega_{\xi})=H_{\xi,\xi'}$. The decomposition $H_{\xi} = H_{\xi, \xi'} \cdot H_{\xi}^{u}$ follows. As $H_{\xi}^{u}$ contains only elliptic elements, the intersection $H_{\xi, \xi'} \cap H_{\xi}^{u}$ is compact. Again by the transitivity of $H_{\xi}$ on $\Opp(\xi)$ one can conclude that $H_{\xi}^{u}$ acts transitively on $\Opp(\xi)$ as well.
\end{proof}

\begin{definition}[see \cite{CM} Definition 3.5]
\label{def::radial_seq}
Let $G$ be a group acting by isometries on a proper \cat space $X$. A sequence $\{g_n\}_{n\geq 1} \subset G$ is called \textbf{radial} in $G$ if there exist an ideal point $\xi \in \partial_{\infty}X$, called \textbf{the radial limit point} of $\{g_n\}_{n\geq 1}$, and a geodesic ray $r : [0, \infty) \to X$ pointing towards $\xi$ such that the sequence $g_n^{-1}(r(0))$ converges to $\xi$, with respect to the cone topology on $X \cup \partial_{\infty} X$, while remaining at bounded distance from $r$. 
\end{definition}

One of the properties of radial sequences is the existence of a limiting bi-infinite geodesic line in $X$ with specific features. The following technical lemma expresses exactly that, although its proof is very easy and relays on the properness of $X$ implying that $X \cup \partial_{\infty} X$ is a compact space, with respect to the cone topology. \begin{lemma}
\label{lem::limit_point_radial}
Let $G$ be a group acting by isometries on a proper \cat space $X$. Suppose $\{g_n\}_{n\geq 1} \subset G$ is a radial sequence with radial limit point $\xi \in \partial_{\infty} X$ and geodesic ray $r : [0, \infty) \to X$ pointing towards $\xi$. Then one can extract a subsequence $\{g_{n_k}\}_{k \geq 1}$ of $\{g_n\}_{n\geq 1}$ such that the sequence of geodesic rays $\{g_{n_k}(r)\}_{k\geq 1}$ converges to a bi-infinite geodesic line $\ell : \RR \to X$ with the properties that $\lim\limits_{k \to \infty}g_{n_k}(r(0))= \ell(-\infty)=:\eta_{-} \in \partial_{\infty} X$ and  $\lim\limits_{k \to \infty}g_{n_k}(\xi)= \ell(\infty)=:\eta_{+} \in \partial_{\infty} X$.
\end{lemma}
\begin{proof}
By the definition of a radial sequence, we have that $\{g_n^{-1}(r(0))\}_{n \geq 1}$ converges to $\xi$ and it stays at bounded distance from $r$. Denote by $r(t_n)$ the projection of $g_n^{-1}(r(0))$  to $r$. Then, by the construction, the sequence $\{g_n(r(t_n))\}_{n\geq 1}$  stays at bounded distance from $r(0)$. As $X$ is proper, implying the visual boundary $\partial_{\infty} X$ is compact, one can extract a subsequence $\{n_k\}_{k \geq 1}$ such that $\lim\limits_{k \to \infty}g_{n_k}(\xi)= \eta_{+} \in \partial_{\infty} X$. Then, up to further extracting a subsequence from $\{n_k\}_{k \geq 1}$, the sequence of points $\{g_{n_k}(r(0))\}_{k\geq 1}$ converges to an ideal point $\eta_{-} \in \partial_{\infty} X$. As $g_{n_k}(r)$ are geodesics rays, from the above limiting procedures one can conclude that the sequence of geodesic rays $\{g_{n_k}(r)\}_{k \geq 1}$ converges to a bi-infinite geodesic line $\ell : \RR \to X$ with $\ell(-\infty)=\eta_{-}$ and $\ell(\infty)=\eta_{+}$.
\end{proof}

\begin{definition}
\label{def::limit_point}
In the hypotheses of Lemma \ref{lem::limit_point_radial}, let $\{g_n\}_{n\geq 1} \subset G$ be a radial sequence with radial limit point $\xi \in \partial_{\infty} X$ and geodesic ray $r : [0, \infty) \to X$ pointing towards $\xi$. An ideal point $\eta \in \partial_{\infty} X$ such that there is a subsequence $\{n_k\}_{k\geq 1}$ with $\lim\limits_{k \to \infty}g_{n_k}(r(0))= \eta \in \partial_{\infty} X$ is called a \textbf{limit point} of $\{g_n\}_{n\geq 1}$.
\end{definition}

The following, also technical lemma, gives a reason why and how radial sequences can be used in applications. 

\begin{lemma}
\label{lem::radial_seq_prop}
Let $X$ be a proper \cat space and $G \leq Isom(X)$ be a closed group. Suppose $\{g_n\}_{n\geq 1} \subset G$ is a radial sequence with radial limit point $\xi \in \partial_{\infty} X$ and geodesic ray $r : [0, \infty) \to X$ pointing towards $\xi$. Suppose also the sequence of geodesic rays $\{g_{n}(r)\}_{n\geq 1}$ converges to a bi-infinite geodesic line $\ell : \RR \to X$ with the properties that $\lim\limits_{n \to \infty}g_{n}(r(0))= \ell(-\infty)=:\eta_{-} \in \partial_{\infty} X$ and  $\lim\limits_{n \to \infty}g_{n}(\xi)= \ell(\infty)=:\eta_{+} \in \partial_{\infty} X$. Let $g \in G_{\xi}$. Then the sequence $\{ g_{n}gg_{n}^{-1}\}_{n\geq 1}$ is bounded and for any of its convergent subsequence $\{ g_{n_k}gg_{n_k}^{-1}\}_{n\geq 1}$ the corresponding limit $\lim\limits_{k \to \infty} g_{n_k}gg_{n_k}^{-1}=:h \in G$ is an element of $G_{\eta_{-},\eta_{+}}$.
\end{lemma}
\begin{proof}
As $g \in G_{\xi}$ and by the definition of the visual boundary $\partial_{\infty} X$, we first notice that the geodesic ray $g(r)$ is  asymptotically equivalent to $r$, and so at bounded distance from the geodesic ray $r$. As $\{g_{n}^{-1}(r(0))\}_{n\geq 1}$ remains at bounded distance from $r$, we have that indeed $d(g_{n}gg_{n}^{-1}(r(0)), r(0))=d(gg_{n}^{-1}(r(0)), g_{n}^{-1}(r(0)))$ is a bounded sequence. Take now a convergent subsequent $\lim\limits_{k \to \infty} g_{n_k}gg_{n_k}^{-1}=:h \in G$. By the hypothesis $\lim\limits_{n \to \infty}g_{n}(\xi)= \eta_{+}$ and by the continuity of the action of $G$ on $\partial_{\infty} X$, we also have $\lim\limits_{k \to \infty} g_{n_k}gg_{n_k}^{-1}(g_{n_k}(\xi))=\lim\limits_{k \to \infty} g_{n_k}g(\xi)=\lim\limits_{k \to \infty} g_{n_k}(\xi)=\eta_{+}=\lim\limits_{k \to \infty} h(g_{n_k}(\xi))=h(\eta_{+})$. Now using $\lim\limits_{n \to \infty}g_{n}(r(0))= \eta_{-}$ and the same computation as above, we have $\lim\limits_{k \to \infty} g_{n_k}gg_{n_k}^{-1}(g_{n_k}(r(0)))=\lim\limits_{k \to \infty} g_{n_k}g(r(0))=\eta_{-}=\lim\limits_{k \to \infty} h(g_{n_k}(r(0)))=h(\eta_{-})$.

Notice, as each element $g_{n}gg_{n}^{-1}$ inherits the same properties as the element $g$ (i.e., elliptic if $g$ is elliptic, hyperbolic of the same translation length if $g$ is hyperbolic), the same is true for a limiting element $h$. 
\end{proof}

\begin{lemma}
\label{lem::radial_points}
Let $\Delta$ be a locally finite thick affine building. Let $H$ be a closed subgroup of $\Aut_{0}(\Delta)$ that acts strongly transitively on $\Delta$. Then every $\xi \in \partial_{\infty}\Delta$ is a radial limit point for some radial sequence in $H_{\xi}$.
\end{lemma}
\begin{proof}
Let $\xi \in \partial_{\infty} \Delta$ and choose an apartment $\Sigma$ of $\Delta$ such that $\xi \in \partial_{\infty} \Sigma$. Let $r : [0, \infty) \to \Sigma$ be a geodesic ray with endpoint $\xi$ and $r(0)$ a special vertex of $\Sigma$. Notice, there are an infinite number of special vertices of the same type as $r(0)$ that are at bounded distance from $r$, and converging to the endpoint $\xi$. By using the affine Weil group $W$ of $\Delta$ as in the proof of Lemma \ref{lem::cocompact1},  for each special vertex $v$ of $\Sigma$, of the same type as $r(0)$, there is a translation automorphism $g_v \in \Stab_{H}(\Sigma)$ sending $r(0)$ to $v$. By choosing a sequence of special vertices $\{v_n\}_{n \geq 0}$, of the same type as $r(0)$, at bounded distance from $r$, and converging to $\xi$, (the inverses of) their corresponding translation automorphisms $\{g_{v_n}\}_{n\geq 1} \in \Stab_{H}(\Sigma)$ will do the job.
\end{proof}

We are now ready to state the Levi decomposition for subgroups $N \leq H_{\xi}$, where $H \leq \Aut(\Delta)$ and $\xi \in \partial_{\infty} \Delta$. This is \cite[Theorem 3.12]{CM} specialised to the case of an affine building $\Delta$.

\begin{theorem}(Levi decomposition II)
\label{thm::levi2}
Let $\Delta$ be a locally finite thick affine building and $H$ a closed subgroup of $\Aut_{0}(\Delta)$ that acts strongly transitively on $\Delta$. Let $\xi \in \partial_{\infty} \Delta$.  Let $N \leq H_{\xi}$ be a closed subgroup that is normalised by some radial sequence $\{g_n\}_{n \geq 1} \subset H_{\xi}$ with radial limit point $\xi$ and unique limit point $\xi_{-}$; hence $\xi_{-} \in \Opp(\xi)$. Then, writing $N^{u}:= N \cap H_{\xi}^{u}$, we have the decomposition 
$$N = N_{\xi_{-}} \cdot N^{u}$$
which is almost semi-direct in the sense that $N_{\xi_{-}} \cap N^{u}$ is compact. In particular, $N^{u}$ acts transitively on  the $N$-orbits of $\xi_{-}$ in $\Opp(\xi)$.
\end{theorem}
\begin{proof}
For the proof see \cite[Theorem 3.12]{CM}. Here we only give the main idea and the intuition behind the proof. As the radial sequence $\{g_n\}_{n \geq 1}$ has the property that belongs to $H_{\xi}$ we have $\lim\limits_{n \to \infty}g_{n}(r(0))= \xi_{-} \in \partial_{\infty} X$ and  $\lim\limits_{n \to \infty}g_{n}(\xi)=\xi$. Then for every $g \in N$, by Lemma \ref{lem::radial_seq_prop} applied to $\{g_n\}_{n \geq 1}$, any of the convergent subsequence $\{ g_{n_k}gg_{n_k}^{-1}\}_{n\geq 1}$ has the corresponding limit $\lim\limits_{k \to \infty} g_{n_k}gg_{n_k}^{-1}=h$ in  $G_{\xi_{-},\xi}$. As $N$ is normalised by the radial sequence $\{g_n\}_{n \geq 1}$, we have moreover that $\lim\limits_{k \to \infty} g_{n_k}gg_{n_k}^{-1}=h \in N_{\xi_{-}}$. Because the properties of $g$ are preserved by each of the elements $g_{n_k}gg_{n_k}$, then also the limit $h$ inherits the properties of $g$ (i.e., elliptic if $g$ is elliptic, hyperbolic of the same translation length if $g$ is hyperbolic). This intuitively implies that the \textbf{horoaction} $\omega_{\xi}: H_{\xi} \to Isom(P(\xi, \xi{-}))$, defined in \cite[Definition 3.1]{CM}, restricted to $N$ has the property that $\omega_{\xi}(N)=\omega_{\xi}(N_{\xi_{-}})$. This will then give the decomposition $N = N_{\xi_{-}} \cdot N^{u}$, where $N^{u}:= N \cap H_{\xi}^{u}= N \cap \Ker(\omega_{\xi})$.
\end{proof}

\begin{proof}[Proof of Theorem \ref{main_thm2}]
The proof is the same as in \cite[Section 7.G]{CM1}. Still, because we work with subgroups $H \leq \Aut_{0}(\Delta)$, and not only with $\Aut_{0}(\Delta)$, we need to treat the general case of $\Delta$ being of any dimension $\geq 2$. In \cite[Section 7.G]{CM1}, because they work with $\Aut_{0}(\Delta)$, they can reduce directly to the dimension of $\Delta$ being $2$ by using the Theorem of Tits that every irreducible, locally finite thick affine building of dimension $\geq 3$ is Moufang (see \cite[Chapter 7]{AB}).

Let $\alpha$ be a root of $\partial_{\infty} \Delta$ and for what follows fix an apartment $\Sigma$ in $\Delta$ such that $\alpha \subset \partial_{\infty} \Sigma$. Then the pointwise stabiliser in $H$ of the ideal half-apartment $\alpha$ is denoted by $H_{\alpha}:=\{ g \in H \; \vert \; g(\alpha)=\alpha \text{ pointwise}\}$.  Then it is easy to see that $H_{\alpha}$ is a closed subgroup of $H$. Because $H$ is strongly transitive on $\Delta$ implying $H$ is strongly transitively on $\partial_{\infty} \Delta$ by \cite[Sec. 17.1]{Ga}, one has $H_{\alpha}$ acts transitively on $\mathcal{A}_{\partial_{\infty} \Delta}(\alpha)$,  the set of all apartments of $\partial_{\infty} \Delta$ that contain $\alpha$ as a half-apartment. We claim  $H_{\alpha} \cap U_{\alpha}(H)$ contains a subgroup that is still transitive on $\mathcal{A}_{\partial_{\infty} \Delta}(\alpha)$, where $U_{\alpha}(H)$ is the root group in $H$ associated with $\alpha$. This claim will definitely imply that $H$ is Moufang.  

It remains to prove the claim. Indeed, as $\partial_{\infty} \Delta$ is a spherical building, the number of chambers of $\alpha$ is finite, and implicitly the number of panels $P$ in $\alpha - \partial \alpha$ is also finite, say $N \geq 1$. We enumerate them by $P_1,P_2,...,P_{N}$ and for each $i \in \{1,...,N\}$, we choose an ideal point $\xi_{i}  \in \partial_{\infty} \Delta$ that is in the interior of the panel $P_{i}$.  By Remark \ref{rem::unipotent_radical}, for every $i \in \{1,...,N\}$, the closed subgroup $H_{\xi_i}^{u}$ fixes pointwise any chamber of $\partial_{\infty} \Delta$ that contains $\xi_{i}$, and equivalently that contains $P_{i}$ as a panel. Thus,  for every $i \in \{1,...,N\} $, the group $H_{\xi_i}^{u}$ fixes pointwise $\St_{\partial_{\infty} \Delta}(P_i)$. By Lemma \ref{lem::radial_points}, each $\xi_{i}$, for $i \in \{1,...,N\}$, is a radial limit point and we denote by $\{t_{i,n}\}_{n \geq 1}$ an associated radial sequence in $H_{\xi_i}$. By the proof of Lemma \ref{lem::radial_points} one can choose $\{t_{i,n}\}_{n \geq 1} \subset \Stab_{H}(\Sigma)$ in such a way that they are moreover translation automorphisms of $\Sigma$, for each $i \in \{1,...,N\}$. In particular, $\{t_{i,n}\}_{n \geq 1}$ normalises $H_{\xi_i}^{u}$, fixes pointwise the ideal boundary $\partial_{\infty} \Sigma$,  thus normalises $H_{\alpha}$ as well. Notice, $H_{\alpha} \leq \bigcap\limits_{i=1}^{N} H_{\xi_i}$.

We denote $H_{\alpha}=:V_{0}$, and for every $i \in \{1,...,N\}$ we put, recursively, $V_{i}:= V_{i-1} \cap H_{\xi_i}^{u}$. We have that $V_{0} \geq V_{1} \geq V_{2} \geq... \geq V_{N}$ and they are all closed subgroups of $H$. Then, by construction we have $V_{N}$ is a subgroup of $U_{\alpha}(H)$. We will prove by induction on $i \in \{1,...,N\}$ that each of $V_{i}$ acts transitively on $\mathcal{A}_{\partial_{\infty} \Delta}(\alpha)$, implying that $V_{N}$ is a subgroup of $U_{\alpha}(H)$ that acts transitively on $\mathcal{A}_{\partial_{\infty} \Delta}(\alpha)$. 

Take $i=1$. Then the radial sequence $\{t_{1,n}\}_{n \geq 1} \subset H_{\xi_1}$, with radial limit point $\xi_{1}$ normalises both $H_{\alpha} =V_{0}$ and $H_{\xi_1}^{u}$, and so also the intersection $V_{1}= V_{0} \cap H_{\xi_1}^{u}$. Then by Theorem \ref{thm::levi2} applied to $V_0$ (up to extraction a subsequence) we have $V_{0}= V_{0,\xi_{-1}} \cdot V_{1}$, where $\xi_{-1} \in \Opp(\xi_1) \cap \partial_{\infty} \Sigma$, and $V_1$ acts transitively on the $V_0$-orbits of $\xi_{-1}$ in $\Opp(\xi_1)$. Since $V_0$ acts transitively on $\mathcal{A}_{\partial_{\infty} \Delta}(\alpha)$, and since there is only one apartment in $\mathcal{A}_{\partial_{\infty} \Delta}(\alpha)$ which contains $\xi_{-1}$ and $\alpha$, one deduces $V_{1}$ acts transitively on $\mathcal{A}_{\partial_{\infty} \Delta}(\alpha)$ as well.

For $2 \leq i \leq N$ we do the same as above by replacing the index $1$ with $i$, and the index $0$ by $i-1$: The radial sequence $\{t_{i,n}\}_{n \geq 1} \subset H_{\xi_i}$, with radial limit point $\xi_{i}$ normalises $H_{\alpha}$ and $H_{\xi_j}^{u}$, for each $j \in \{1,...,N\}$, thus $V_{i-1}$ (by induction step), and so also the intersection $V_{i}= V_{i-1} \cap H_{\xi_i}^{u}$. Then by Theorem \ref{thm::levi2} applied to $V_{i-1}$ (up to extraction a subsequence) we have $V_{i-1}= V_{i-1,\xi_{-i}} \cdot V_{i}$, where $\xi_{-i} \in \Opp(\xi_i) \cap \partial_{\infty} \Sigma$, and $V_i$ acts transitively on the $V_{i-1}$-orbits of $\xi_{-i}$ in $\Opp(\xi_i)$. Since $V_{i-1}$ acts transitively on $\mathcal{A}_{\partial_{\infty} \Delta}(\alpha)$ (by the induction step), and since there is only one apartment in $\mathcal{A}_{\partial_{\infty} \Delta}(\alpha)$ which contains $\xi_{-i}$ and $\alpha$, one deduces $V_{i}$ acts transitively on $\mathcal{A}_{\partial_{\infty} \Delta}(\alpha)$ as well. The conclusion of the theorem follows.
\end{proof}

\section{Proof of Theorem \ref{thm::main3}}
Let $\Delta$ be an irreducible, locally finite thick affine building of dimension $\geq 2$ such that $\Aut_{0}(\Delta)$ is strongly transitive on $\Delta$, thus also on $\partial_{\infty} \Delta$ by \cite[Sec. 17.1]{Ga}.  For what follows, the irreducibility cannot be removed as noticed after the Corollary \ref{cor::simply_tran}. By Theorem \ref{main_thm2} $\Aut_{0}(\Delta)$ is Moufang, and so is the spherical building $\partial_{\infty} \Delta$ (in the sense of \cite[Definition 7.27]{AB}). Fix an apartment $\Sigma$ of $\partial_{\infty} \Delta$ and an ideal chamber $c_{+} \in \Ch(\Sigma)$ with its opposite chamber $c_{-} \in \Ch(\Sigma)$.  Denote by $\Phi$ the set of all roots of $\Sigma$ and by $\Phi^{+}$ the roots in $\Phi$ that contain $c$ as a chamber; we call $\Phi^{+}$ the set of positive roots with respect to $c_{+}$. We set $$U^{+}:= \langle U_{\alpha}(\Aut_{0}(\Delta)) \; \vert \; \alpha \in \Phi^{+}\rangle \text{ and } U^{-}:= \langle U_{\alpha}(\Aut_{0}(\Delta)) \; \vert \; -\alpha \in \Phi^{+}\rangle.$$

\begin{remark}
\label{rem::coro_2.3}
For any closed and strongly transitive subgroup $H$ of $\Aut_{0}(\Delta)$, Theorem \ref{main_thm2} and Corollary \ref{cor::equal_root_groups} tell us that we obtain the same groups $U^{+}$, resp., $U^{-}$ if we replace  $U_{\alpha}(\Aut_{0}(\Delta))$ by $ U_{\alpha}(H)$. 
\end{remark}

Because we work with the spherical building $\partial_{\infty} \Delta$ which is Moufang, we can apply the results of \cite[Chapter 6 and Thm. 6.18]{Ro}. We recall the notation. Let $$B^{+}:=\Stab_{\Aut_{0}(\Delta)}(c_{+})  \; \text{ and } \; B^{-}:=\Stab_{\Aut_{0}(\Delta)}(c_{-})$$ be the positive, resp., negative,  minimal parabolic subgroup of $\Aut_{0}(\Delta)$. 

For a closed subgroup $H \leq \Aut_{0}(\Delta)$ we set $$B(H)^{+}:= \Stab_{H}(c_{+}) \leq B^{+} \; \text{and} \;  B(H)^{-}:= \Stab_{H}(c_{-}) \leq B^{-}.$$

\begin{proposition}
\label{prop::normal_U}
Let  $H$ be a closed and strongly transitive subgroup of $\Aut_{0}(\Delta)$. Then  $B^{\pm} =  U^{\pm} \rtimes M$ and $B(H)^{\pm} = U^{\pm} \rtimes M(H)$, where $M:= B^{+} \cap B^{-}$ and $M(H):= B(H)^{+} \cap B(H)^{-}$. 
In particular, $U^+$ (resp., $U^{-}$) is a normal subgroup in both subgroups  $B^{+}$ and  $B(H)^{+}$ (resp., $B^{-}$ and $B(H)^{-}$).
\end{proposition}
\begin{proof}
This follows directly from \cite[Chapter 6, Thms. 6.17 and 6.18]{Ro} and  Corollary \ref{cor::equal_root_groups} above.
\end{proof}

\begin{proposition}
\label{prop::normal_U_pm}
Let  $H$ be a closed and strongly transitive subgroup of $\Aut_{0}(\Delta)$. Then $\overline{\langle U^{+}, U^{-} \rangle}$ is normal in $H$ and $\langle U^{+}, U^{-} \rangle= H^{+}$.
\end{proposition}
\begin{proof}
Let us prove first that $\overline{\langle U^{+}, U^{-} \rangle}$ is normal in $H$.  Indeed, by \cite[Prop. 4.11]{Cio} we know $H=\langle H_{c_{+}}^{0},H_{c_{-}}^{0}\rangle$, where $H_{c_{\pm}}^{0}= (B(H)^{\pm})^{0}:=\{ g \in B(H)^{\pm} \; \vert \; g \text{ elliptic} \}$. Notice $U^{\pm} \leq H_{c_{\pm}}^{0}$ and by Proposition \ref{prop::normal_U} we have  $U^{\pm}$ is normal in $H_{c_{\pm}}^{0}$. To prove $\overline{\langle U^{+}, U^{-} \rangle}$ is normal in $H$ it is enough to show that for every $g \in H_{c_{\pm}}^{0}$ we have $g U^{\pm}g^{-1} \leq \langle U^{+}, U^{-} \rangle$. As $H_{c_{\pm}}^{0}= U^{\pm} \rtimes M(H)^{0}$, where  $M(H)^{0}:=\{g \in M(H) \; \vert \; g \text{ elliptic} \}$, it is enough to verify $g U^{\pm}g^{-1} \leq \langle U^{+}, U^{-} \rangle$ only for $g \in M(H)^{0}$ and this follows easily because $U^{\pm}$  is normal in $H_{c_{\pm}}^{0}$.

Let us now prove that $\langle U^{+}, U^{-} \rangle= H^{+}$. Indeed, it is trivial that $\langle U^{+}, U^{-} \rangle \leq H^{+}$. Thus it remains to show that for every root $\alpha$ of $\partial_{\infty}\Delta$ we have $U_{\alpha} \leq \langle U^{+}, U^{-} \rangle$. To do that we first prove that $U^{+}$ acts transitively on the set of all chambers $c \in \Ch(\partial_{\infty} \Delta)$ that are opposite $c_+$. In that respect, it is enough to show that for every chamber $c \in \Ch(\partial_{\infty} \Delta)$ opposite $c_{+}$ there is an element $g \in U^{+}$ such that $g(c) = c_{-}$. By taking an apartment $\Sigma_{1} \in \partial_{\infty} \Delta$ containing the chambers $c$ and $c_{+}$, and using the transitivity properties of the root groups $U_{\beta}(H)=U_{\beta}(\partial_{\infty} \Delta)$ for $\beta \in \Phi^{+}$, one can send one by one the chambers of a gallery in $\Sigma_1$ from $c$ to $c_{+}$ to chambers in $\Sigma$. This will produce an element $g \in U^{+}$ with $g(c)=c_{-}$.

Secondly, notice that by the properties of the root groups (see for example \cite[Chapter 6, Section 4, property (M3) and Prop. (6.14)]{Ro}) we have that $\langle U^{+}, U^{-} \rangle$ acts transitively on the set of all chambers of $\Sigma$. In particular, the spherical Weyl group associated with $\partial_{\infty} \Delta$ is contained in $\langle U^{+}, U^{-} \rangle$. 

Thirdly, we prove that $\langle U^{+}, U^{-} \rangle$ acts strongly transitively on $\partial_{\infty} \Delta$. For that, it is again enough to show that for every apartment $\Sigma_1$ in $\partial_{\infty} \Delta$ and every chamber $c_1 \in \Ch(\Sigma_1)$ there is $g \in \langle U^{+}, U^{-} \rangle$ such that $g(\Sigma_1)=\Sigma$ and $g(c_1)=c_{+}$. Then take an apartment $\Sigma_2$ in $\partial_{\infty} \Delta$ with $c_1,c_{+} \in \Ch(\Sigma_2)$. By the first step above and considering the chamber in $\Sigma_2$ opposite $c_{+}$, there is an element $g_2 \in \langle U^{+}, U^{-} \rangle$ with $g_2(\Sigma_2)=\Sigma$ and $g_2(c_{+})=c_{+}$, so $g_{2}(c_1) \in \Ch(\Sigma)$. By applying again the first step in combination with the second one, there is an element $g_1 \in \langle U^{+}, U^{-} \rangle$ with $g_1(g_2(c_1))=g_2(c_1)$ and $g_1(g_2(\Sigma_1))=\Sigma$. To conclude this third step we then apply the second step in order to obtain $g_0 \in \langle U^{+}, U^{-} \rangle$ with $g_0(\Sigma)=\Sigma$ (not pointwise) and $g_0(g_1(g_2(c_1)))=g(g_2(c_1))=c_{+}$. Then $g:=g_0g_1g_2$ will do the job.

To conclude the proof of the second part of the proposition, we use the strongly transitive action of $\langle U^{+}, U^{-} \rangle$ on $\partial_{\infty} \Delta$ and the fact that $g U_{\alpha} g^{-1}=U_{g\alpha}$ for any root $\alpha$ of $\partial_{\infty} \Delta$ and any $g \in \Aut_{0}(\Delta)$ (see \cite[Lemma 7.25(1)]{AB}).
\end{proof}

\begin{proof}[Proof of Theorem \ref{thm::main3}]

By Theorem \ref{main_thm2} $\Aut_{0}(\Delta)$ is Moufang. As moreover $\Delta$ is irreducible, by Theorem \ref{main_thm22} there exist a non-Archimedean local field $k$ and a semisimple, (absolutely) simple algebraic group $\G$ over $k$ and of $k$-rank $\geq 2$ such that $G= \G(k)$ acts by type-preserving automorphisms and strongly transitively on $\Delta$.  Let $H \leq \Aut_{0}(\Delta)$ be a topologically simple, closed and strongly transitive subgroup. By Theorem \ref{main_thm2} the group $H$ is Moufang also. Then by Corollary \ref{cor::equal_root_groups}, for every root $\alpha$ of $\partial_{\infty} \Delta$ we have $U_{\alpha}(\Aut_{0}(\Delta))= U_{\alpha}(H)$. By Remark \ref{rem::coro_2.3} and Proposition \ref{prop::normal_U_pm} we have $\overline{\langle U^{+}, U^{-} \rangle}$ is normal in $H$. As $H$ is topologically simple and $\overline{\langle U^{+}, U^{-} \rangle}$ is closed, we obtain $H = \overline{\langle U^{+}, U^{-} \rangle}=\overline{H^{+}}$.  By replacing $H$ with $G$, we have $\overline{\langle U^{+}, U^{-} \rangle}$ is normal in $G$ and so also in $G^{+}$. By \cite[Theorem 1.5.6]{Ma} one conclude that $G^{+}$ is contained in $\overline{\langle U^{+}, U^{-} \rangle}$, which is a subgroup of $G^{+}$ by definition (recall $G^+$ is closed and normal in $G$). Thus $G^{+} = \overline{\langle U^{+}, U^{-} \rangle}$, so $G^{+}=H$ and we are done.

\end{proof}

\section{Proof of Theorem \ref{thm::main}}
\label{sec::main_thm}
In this section we work in the setting where $\Delta$ is a locally finite thick affine building and $G$ a  closed subgroup of $\Aut_{0}(\Delta)$ that acts strongly transitively on $\Delta$. Fix for what follows an apartment $\mathcal{A}$ of $\Delta$ and a special vertex $x_0 \in \mathcal{A}$. 

Let $\alpha:= \{a_n\}_{n \geq 1}$ be a sequence in $G$ that escapes every compact subset of $G$. Then, as the space $\Delta \cup \partial_{\infty} \Delta$ is compact with respect to the cone topology, the sequence $\{a_n\}_{n \geq 1}$ admits a subsequence $\beta:=\{a_{n_k}\}_{k\geq 1}$ with the property that $\{a_{n_k}(x_0)\}_{k\geq 1}$ converges to an ideal point $\xi \in \partial_{\infty} \Delta$. Notice, there is a unique minimal ideal simplex $\sigma \subset \partial_{\infty} \Delta$ such that $\xi$ is in the interior of $\sigma$. One can have anything in-between  $\sigma$ is an ideal chamber or $\sigma$ is an ideal vertex of the spherical building $\partial_{\infty} \Delta$. Recall $G_{\sigma}:=\{ g \in G \; \vert \; g(\sigma)=\sigma\}$ is called a \textbf{parabolic subgroup} of $G$ and it is closed. $G_{\sigma}$ is maximal when $\sigma$ is an ideal vertex in the spherical building $\partial_{\infty} \Delta$, and minimal when $\sigma$ is an ideal chamber of the spherical building $\partial_{\infty} \Delta$. 

\begin{definition}
\label{def::st_cone_N}
Let $\sigma$ be an ideal simplex in $\partial_{\infty} \Delta$ and choose an apartment $\mathcal{B}$ of $\Delta$  such that $\sigma \subset \partial_{\infty} \mathcal{B}$. Let $x \in \mathcal{B}$ be a point. We set $\st(\sigma,\partial_{\infty} \mathcal{B}):= \{c \in \Ch(\partial_{\infty} \mathcal{B}) \; \vert \; \sigma \subset c\}$, resp., $\st(x, \mathcal{B}):= \{c \in \Ch(\mathcal{B}) \; \vert \; x \text{ in the closure of } c\}$, and call them the \textbf{$\partial_{\infty} \mathcal{B}$-star of $\sigma$}, resp., the \textbf{$\mathcal{B}$-star of $x$}.  Notice, $\st(\sigma,\partial_{\infty} \mathcal{B})$, resp., $\st(x, \mathcal{B})$, are both a finite set of ideal chambers, resp., chambers. As well,  the associated \textbf{$\sigma$-cone in $\mathcal{B}$ with base-point $x \in \mathcal{B}$} is defined by $Q(\sigma, \mathcal{B},x):= \bigcup\limits_{c \in \st(\sigma,\partial_{\infty} \mathcal{B})} Q(c,\mathcal{B},x)$, where $Q(c,\mathcal{B},x) \subset \mathcal{B}$ is the Weyl sector emanating from the point $x$ and having $c$ as ideal chamber.  A  \textbf{$\sigma$-cone in $\mathcal{B}$} is  $Q(\sigma,\mathcal{B},x)$ for some $x \in \mathcal{B}$. Then we take 
$$G_{\st(\sigma,\partial_{\infty} \mathcal{B})}^{0}:=\{ g \in G \; \vert \; g \text{ fixes pointwise a  $\sigma$-cone in } \mathcal{B}\}$$
and it is easy to see this is a closed subgroup of $G$. Notice, the intersection of two $\sigma$-cones in $\mathcal{B}$ contains a $\sigma$-cone in $\mathcal{B}$.
\end{definition}

Notice that Definition \ref{def::st_cone_N} is a weakening of Definition \ref{def::unipotent_radical}.

Let $\alpha:= \{a_n\}_{n \geq 1}$ be a sequence in $G$ and associated with it consider the set 
$$U^{+}_{\alpha}:=\{g \in G \; \vert \; \lim\limits_{n \to \infty}  a_{n}^{-1}g a_{n}=e \}$$
that is a subgroup of $G$, called the \textbf{positive contraction group of $\alpha$}. In the same way, but using $a_{n}g a_{n}^{-1}$ one defines  the \textbf{negative contraction group} $U^{-}_{\alpha}$ of $\alpha$. In general the contraction groups are not closed in $G$.

\begin{lemma}
\label{lem::contr_group_alpha}
Let $\mathcal{B}$ be an apartment of $\Delta$. Let $\alpha:= \{a_n\}_{n \geq 1}$ be a sequence in $G$ and $g$ a hyperbolic element in $ \Stab_{G}(\mathcal{B})$. Assume $\{a_{n}(x_0)\}_{n \geq 1} \subset \mathcal{B}$  converges to an ideal point $\xi \in \partial_{\infty} \mathcal{B}$ contained in the interior of the ideal simplex $\sigma \subset \partial_{\infty} \mathcal{B}$. 

Then $U^{+}_{\alpha} \leq G_{\st(\sigma,\partial_{\infty} \mathcal{B})}^{0} \leq G_{\sigma}$, and $G_{\st(\sigma,\partial_{\infty} \mathcal{B})}^{0}$ is normalized by $g$, i.e.,  $gG_{\st(\sigma,\partial_{\infty} \mathcal{B})}^{0}g^{-1}= G_{\st(\sigma,\partial_{\infty} \mathcal{B})}^{0}$.
\end{lemma}

\begin{proof} It is clear  $G_{\st(\sigma,\partial_{\infty} \mathcal{B})}^{0} \leq G_{\sigma}$. To prove $U^{+}_{\alpha} \leq  G_{\st(\sigma,\partial_{\infty} \mathcal{B})}^{0} $ one just needs to apply the definition of $U^{+}_{\alpha}$, to notice that each element of $U^{+}_{\alpha}$ fixes point-wise larger and larger balls in $\Delta$ that are centered in $\{a_{n}(x_0)\}_{n \geq 1}$, and by the continuity of the $G$-action on $\partial_{\infty} \Delta$, that element of $U^{+}_{\alpha}$  stabilizes $\sigma \subset \partial_{\infty} \mathcal{B}$. Then if an element $h \in G$ fixes point-wise the $\mathcal{B}$-star of a special vertex $x$ in $\mathcal{B}$ and stabilizes the ideal simplex $\sigma \subset \partial_{\infty} \mathcal{B}$, then $h$ fixes point-wise the combinatorial convex hall in $\mathcal{B}$ determined by $\st(x,\mathcal{B})$ and $\sigma$, and this must contain the $\sigma$-cone $Q(\sigma,\mathcal{B},x)$.

Now let us prove $gG_{\st(\sigma,\partial_{\infty} \mathcal{B})}^{0}g^{-1}= G_{\st(\sigma,\partial_{\infty} \mathcal{B})}^{0}$. Indeed, as $g$ is a hyperbolic element in $ \Stab_{G}(\mathcal{B})$, then $g (\partial_{\infty} \mathcal{B})= \partial_{\infty} \mathcal{B}$ point-wise. Then it is clear if $h \in G_{\st(\sigma,\partial_{\infty} \mathcal{B})}^{0}$ fixes point-wise the $\sigma$-cone $Q(\sigma, \mathcal{B},x)$, then $g (Q(\sigma, \mathcal{B},x))= Q(\sigma, \mathcal{B}, g(x))$ is again a $\sigma$-cone in $\mathcal{B}$ that is point-wise fixed by $ghg^{-1}$.  Thus $gG_{\st(\sigma,\partial_{\infty} \mathcal{B})}^{0}g^{-1} \subseteq G_{\st(\sigma,\partial_{\infty} \mathcal{B})}^{0}$. In the same way one can prove $g^{-1}G_{\st(\sigma,\partial_{\infty} \mathcal{B})}^{0}g \subseteq G_{\st(\sigma,\partial_{\infty} \mathcal{B})}^{0}$, and we are done.

\end{proof}

The following is a generalization of \cite[Prop. 4.3]{BM00b} to the case of affine buildings, and its proof is the same. From the point of view of finding $H$-invariant vectors of unitary representations of $G$, with $H \leq G$ as large as possible, it gives better results than Mautner's phenomenon. 

\begin{proposition}
\label{prop:BM}
Let $G$ be a closed and type-preserving subgroup of $\Aut(\Delta)$.
Let $\sigma \subset \partial_{\infty} \Delta$ be an ideal simplex, $\mathcal{B}$ an apartment of $\Delta$ with $\sigma \subset \partial_{\infty} \mathcal{B}$, and $x_0 \in \mathcal{B}$ a special vertex. Let $(\pi, \mathcal{H}, \left\langle \cdot , \cdot \right\rangle)$ be a strongly continuous unitary representation of $G_{\sigma}$ and $v,w \in \mathcal{H}$ nonzero unit vectors. Suppose there is a sequence $\{a_n\}_{n \geq 1} \subset \Stab_{G}(\mathcal{B})$  of hyperbolic elements such that $\{a_n(x_0)\}_{n\geq 1}$ converges to an ideal point $\xi$ in the interior of $\sigma$, and $\langle \pi(a_n)(v), w \rangle$ does not vanish at infinity, when $n \to \infty$. Then:
\begin{enumerate}
\item
\label{prop::BM1}
$G_{\st(\sigma,\partial_{\infty} \mathcal{B})}^{0} = C_0 \cup \bigcup\limits_{n \geq 1}a_n C_0 a_n^{-1}$, where $C_0:= G_{x_0} \cap G_{\st(\sigma,\partial_{\infty} \mathcal{B})}^{0}$
\item
\label{prop::BM2}
The subspace $\mathcal{H}^{\infty}:=\{ w \in \mathcal{H} \; \vert \; \Stab_{G_{\st(\sigma,\partial_{\infty} \mathcal{B})}^{0}}(w) \text{ is open in }  G_{\st(\sigma,\partial_{\infty} \mathcal{B})}^{0}\}$ is dense in $\mathcal{H}$
\item
\label{prop::BM3}
 There exists $w_1 \neq 0 \in \mathcal{H}$ such that $\Stab_{G_{\st(\sigma,\partial_{\infty} \mathcal{B})}^{0}}(w_1)$ is of finite index in $G_{\st(\sigma,\partial_{\infty} \mathcal{B})}^{0}$.
\end{enumerate}
\end{proposition}

\begin{proof}
Let us prove \ref{prop::BM1}.. First notice $C_0$ is a compact open subgroup of $G_{\st(\sigma,\partial_{\infty} \mathcal{B})}^{0}$. As every $a_n$ is a hyperbolic element in $ \Stab_{G}(\mathcal{B})$, then $a_n (\partial_{\infty} \mathcal{B})= \partial_{\infty} \mathcal{B}$ point-wise. Also, it is clear $C_0 \cup \bigcup\limits_{n \geq 1}a_n C_0 a_n^{-1} \subseteq G_{\st(\sigma,\partial_{\infty} \mathcal{B})}^{0}$. Let $g \in G_{\st(\sigma,\partial_{\infty} \mathcal{B})}^{0}$ and let $Q(\sigma,\mathcal{B},x)$ be a $\sigma$-cone in $\mathcal{B}$ point-wise fixed by $g$. As $Q(\sigma,\mathcal{B},x)$ is a convex cone and because $\{a_n(x_0)\}_{n\geq 1}$ converges to $\xi \in \sigma$, there exists $N>0$ such that $a_n(x_0) \in Q(\sigma,\mathcal{B},x)$ for every $n \geq N$. Then $Q(\sigma,\mathcal{B},a_n(x_0)) \subset Q(\sigma,\mathcal{B},x)$ for $n \geq N$, and so $g \in a_n C_0 a_n^{-1}$. 

Let us prove \ref{prop::BM2}.. Fix a decreasing sequence $\{K_n\}_{n \geq 1}$ of compact open subgroups of $G_{\st(\sigma,\partial_{\infty} \mathcal{B})}^{0}$ such that $\bigcap\limits_{n \geq 1} K_n= \{e\}$. Such sequence exists. Then the sequence $\{f_n\}_{n \geq 1}$, where $f_n:=\frac{1}{m(K_n)} \cdot \chi_{K_n}$ is an $L^{1}$-approximation of the identity $e$. Hence $\lim\limits_{n \to \infty}\pi(f_n)(w_1) = w_1$ for every $w_1 \in \mathcal{H}$. As $\pi(f_n)(w_1)= \int\limits_{K_n}\pi(k)(w_1) dk \in \mathcal{H}^{\infty}$, for every $w_1 \in \mathcal{H}$, we obtain the conclusion.

Let us prove \ref{prop::BM3}.. As by hypothesis $\langle \pi(a_n)(v), w \rangle$ does not vanish at infinity, when $n \to \infty$, by \ref{prop::BM2}. there  exist nonzero $u_1, u_2 \in \mathcal{H}^{\infty}$ with $\langle \pi(a_n)(u_1), u_2 \rangle$ not vanishing at infinity either. Then there is a subsequence $\{n_i\}_{i \geq 1}$ with $n_i \xrightarrow[i \to \infty]{} \infty $ and $w_1 \in \mathcal{H}$, $w_1 \neq 0$ such that $\{\pi(a_{n_i})(u_1)\}_{i \geq 1}$ converges weakly to $w_1$. 

Let $K \leq C_0$ be an open subgroup (thus of finite index and closed in $G_{\st(\sigma,\partial_{\infty} \mathcal{B})}^{0}$) for which $\pi(k)(u_1)=u_1$, for every $k \in K$. Then $\pi(a_{n_i})(u_1)$ is fixed by $a_{n_i}Ka_{n_i}^{-1}$. By passing to a subsequence, one may assume the sequence of closed subgroups $\{a_{n_i} K a_{n_i}^{-1}\}_{i \geq 1}$ converges in the Chabauty topology  to a closed subgroup $L \leq G_{\st(\sigma,\partial_{\infty} \mathcal{B})}^{0}$. We claim $w_1$ is $L$-invariant and $L$ is of finite index in $G_{\st(\sigma,\partial_{\infty} \mathcal{B})}^{0}$. 

\textbf{Claim 1:} $w_1$ is $L$-invariant. Indeed, let $\ell \in L$. Then there is $m_i \in a_{n_i} K a_{n_i}^{-1}$, for every $i \geq 1$, such that $\lim\limits_{i \to \infty}m_i=\ell$. We have 
\begin{equation}
\label{equ::L_inv}
\begin{split}
\langle \pi(\ell)(w_1), w_1 \rangle - \langle w_1, w_1 \rangle &=(\langle \pi(\ell)(w_1), w_1 \rangle-\langle \pi(\ell)(w_1), \pi(a_{n_i})(u_1) \rangle)\\
&\; \;+(\langle \pi(\ell)(w_1), \pi(a_{n_i})(u_1) \rangle-\langle \pi(m_i)(w_1), \pi(a_{n_i})(u_1) \rangle)\\
&\; \;+(\langle w_1, \pi(a_{n_i})(u_1) \rangle-\langle w_1, w_1 \rangle).\\
\end{split}
\end{equation}
Notice, $m_i \in a_{n_i} K a_{n_i}^{-1}$ implies $\langle \pi(m_i)(w_1), \pi(a_{n_i})(u_1) \rangle= \langle w_1, \pi(a_{n_i})(u_1) \rangle)$. Also the first and the last summand from equality (\ref{equ::L_inv}) converger to zero since $\{\pi(a_{n_i})(u_1)\}_{i \geq 1}$ converges weakly to $w_1$. The second summand also converges to zero since it is bounded by $\vert \vert \pi(\ell)(w_1) - \pi(m_i)(w_1)\vert \vert \cdot \vert \vert u_1 \vert \vert $ and $\lim\limits_{i \to \infty}m_i=\ell$. The Claim 1 follows.

 \textbf{Claim 2:} the subgroup $L$ is of finite index in $G_{\st(\sigma,\partial_{\infty} \mathcal{B})}^{0}$. Indeed, let $d \in \NN$ be the index of $K$ in $C_0$ and pick $x_1,\cdots, x_r \in G_{\st(\sigma,\partial_{\infty} \mathcal{B})}^{0}$ with $r >d$. Since $G_{\st(\sigma,\partial_{\infty} \mathcal{B})}^{0}= \bigcup\limits_{i \geq 1}a_{n_i} C_0 a_{n_i}^{-1}$, there exists $i_0$ such that $\{x_1, \cdots x_r\} \subset a_{n_i} C_0 a_{n_i}^{-1}$ for every $i \geq i_0$. Since $r >d$, for every $i \geq i_0$, there exist $j_i \neq k_i$ such that $x_{j_i}^{-1}x_{k_i} \in a_{n_i} K a_{n_i}^{-1}$. As $r < +\infty$, there exist $j\neq k$ such that $x_{j}^{-1}x_{k} \in a_{n_i} K a_{n_i}^{-1}$ for infinitely many $i \geq i_0$. This implies $x_{j}^{-1}x_{k} \in L$ and shows that the index of $L$ in $G_{\st(\sigma,\partial_{\infty} \mathcal{B})}^{0}$ is at most $d$. This proves the claim and the proposition. 
\end{proof}

\begin{lemma}(See \cite[Lemma 4.12]{Cio})
\label{lem::existance_elliptic}
Let $G$ be a closed, strongly transitive and type-preserving subgroup of $\Aut(\Delta)$, and $\sigma \subset \partial_{\infty} \Delta$ an ideal simplex.  Let $\mathcal{A}, \mathcal{B}$ be two apartments in $\Delta$ such that $\mathcal{A} \cap \mathcal{B}$ is a half-apartment containing $\St(\sigma, \partial_{\infty} \mathcal{A})$ in its boundary at infinity. Then:
\begin{enumerate}
\item
\label{lem::existance_elliptic1}
There is $g \in G_{\st(\sigma,\partial_{\infty} \mathcal{A})}^{0}$ mapping $\mathcal{B}$ to $\mathcal{A}$ and fixing point-wise $\mathcal{A} \cap \mathcal{B}$
\item
\label{lem::existance_elliptic2}
For every wall $M$ of $\mathcal{A}$ there is $g \in G$ that stabilizes $\mathcal{A}$ and acts on $\mathcal{A}$ as the reflection through the wall $M$.
\item
\label{lem::existance_elliptic3}
For every wall $M$ of $\mathcal{A}$ there is a hyperbolic element $g \in \Stab_G(\mathcal{A})$ whose translation axis is perpendicular to $M$.
\end{enumerate}
\end{lemma}

\begin{proof}
Part \ref{lem::existance_elliptic1}. of the lemma follows by the strong transitivity of $G$ on $\Delta$, and by choosing a chamber $c$ in the interior of the  half-apartment $\mathcal{A} \cap \mathcal{B}$. Then there is $g \in G$ such that $g(\mathcal{B})=\mathcal{A}$ and $g(c)=c$. In particular, $g$ fixes point-wise the half-apartment $\mathcal{A} \cap \mathcal{B}$, and as well its ideal boundary containing $\St(\sigma, \partial_{\infty} \mathcal{A})$. This concludes $g \in G_{\st(\sigma,\partial_{\infty} \mathcal{A})}^{0}$.

To prove part \ref{lem::existance_elliptic2}., let $M$ be a wall of $\mathcal{A}$. Let $H$ and $H'$ be the two half-apartments of $\mathcal{A}$ determined by $M$. Since $\Delta$ is thick, there is a half-apartment $H''$ such that $H \cup H''$ and $H' \cup H''$ are both apartments of $\Delta$. By the first part of the lemma there are $u,v,w \in G$ such that $u$ fixes $H$ point-wise and maps $H'$ to $H''$, the elements $v,w$ both fix $H'$ point-wise and  $v(H'')=H, w(H)=u^{-1}(H')$. Set $g=vuw$ and by construction $g$ fixes point-wise the wall $M$ and $g(H)=H', g(H')=H$. 

To obtain the last part of the lemma one just needs to apply \ref{lem::existance_elliptic2}. for example, for two consecutive parallel walls of $\mathcal{A}$. The conclusion follows.
\end{proof}

We also record the following lemma. Still, this is not used further.
\begin{lemma}
\label{lem::iden_elements}
Let $G$ be a closed, strongly transitive and type-preserving subgroup of $\Aut(\Delta)$. Let $\mathcal{A}$ be an apartment of $\Delta$, $\sigma$ a simplex in $\partial_{\infty} \mathcal{A}$,  and $M$ a wall in $\mathcal{A}$.  Suppose $\St(\sigma, \partial_{\infty} \mathcal{A}) \subset \partial_{\infty} H$, where $H$ is one of the half-apartments of $\mathcal{A}$ associated with $M$. Let $r \in \Stab_{G}(\mathcal{A})$ be a reflection on $\mathcal{A}$ through the wall $M$. 

Then for every hyperbolic element $g \in  \Stab_{G}(\mathcal{A})$ of even translation length, and having its translation axis perpendicular to $M$ and its repelling end-point in $\partial_{\infty} H$,  the double coset $G_{\st(\sigma,\partial_{\infty} \mathcal{A})}^{0}rg G_{\st(\sigma,\partial_{\infty} \mathcal{A})}^{0}$ contains elliptic elements that fix point-wise a half-apartment of $\mathcal{A}$ associated with an $M$-parallel wall between $M$ and $g^{-1}(M)$, and not containing $\sigma$ in its ideal boundary.
\end{lemma}

\begin{proof}
Let $g$ be a hyperbolic element in  $\Stab_{G}(\mathcal{A})$ of even translation length, and having its translation axis perpendicular to $M$ and its repelling end-point in $\partial_{\infty} H$. 

As $g$ has even translation length, let $M'$ be the $M$-parallel wall in $\mathcal{A}$ at the half distance between $M$ and $g^{-1}(M)$. Notice the wall $M'$ is in the half-apartment $H$.  Denote by $H'$ the half-apartment of $\mathcal{A} - M'$ having $\St(\sigma, \partial_{\infty} \mathcal{A})$ in its ideal boundary.  By Lemma \ref{lem::existance_elliptic} there is $h \in G_{\st(\sigma,\partial_{\infty} \mathcal{A})}^{0}$ fixing $H'$ point-wise and moving away the half apartment $\mathcal{A} - H'$. Then apply $g$ and so the half-apartment  $gh(\mathcal{A} - H')$ intersects $\mathcal{A}$ only in the wall $g(M')$.  Apply the reflection $r$, and by the above construction $rg(M')=M'$. Apply again an element $h' \in G_{\st(\sigma,\partial_{\infty} \mathcal{A})}^{0}$ that fixes $H'$ point-wise and moves the half-apartment $rgh(\mathcal{A} - H')$ back to $\mathcal{A} - H'$. We have $h'rgh(\mathcal{A} - H')=\mathcal{A} - H'$ point-wise. Notice $h'rgh(H') \neq H'$, and we are done. 
\end{proof}

\begin{lemma}
\label{lem::sequence_elliptic_elements}
Let $G$ be a closed, strongly transitive and type-preserving subgroup of $\Aut(\Delta)$. Let $\mathcal{A}$ be an apartment of $\Delta$ and $\sigma$ a simplex in $\partial_{\infty} \mathcal{A}$. Then there exist a sequence of  half-apartments $\{H_n\}_{n \geq 1}$ of $\mathcal{A}$ and elements $\{k_n\}_{n \geq 1} \subset G_{x_0}$ such that:
\begin{enumerate}
\item
$H_n \subsetneq H_{n+1}$, and $x_0 \in H_n$, for every $n \geq 1$
\item
$\St(\sigma, \partial_{\infty} \mathcal{A}) \subset \partial_{\infty} (\mathcal{A} - H_1)$
\item
$\lim\limits_{n \to \infty}k_n=e$, and for every $n \geq 1$ one has $k_n(H_n)=H_n$ and $k_n(\mathcal{A} - H_n)\cap \mathcal{A}$ is a wall.
\end{enumerate}
\end{lemma}

\begin{proof}
One uses Lemma \ref{lem::existance_elliptic} and the fact that, when $\Delta$ is irreducible, by Theorems \ref{main_thm2} and \ref{thm::main3}, and Proposition \ref{prop::normal_U_pm}, $G$ has many and enough elements.  
\end{proof}

\begin{lemma}
\label{lem::hyper_elements}
Let $G$ be a closed, strongly transitive and type-preserving subgroup of $\Aut(\Delta)$. Let $\mathcal{A}$ be an apartment of $\Delta$, $\sigma$ a simplex in $\partial_{\infty} \mathcal{A}$ and $M$ a wall in $\mathcal{A}$.  Suppose $\St(\sigma, \partial_{\infty} \mathcal{A}) \subset \partial_{\infty} H$, where $H$ is one of the two half-apartments of $\mathcal{A}$ associated with $M$. Let $k \in G$ such that $k(\mathcal{A}- H)=\mathcal{A}- H$ point-wise and $k(H) \cap \mathcal{A}= M$. Then for every $\ell \in \NN^{*}$ there exists $g \in G_{\st(\sigma,\partial_{\infty} \mathcal{A})}^{0}k G_{\st(\sigma,\partial_{\infty} \mathcal{A})}^{0}$ that is hyperbolic, of translation axis perpendicular to $M$, containing $\mathcal{A} - H$ in its translation apartment, of attracting end-point in $\partial_{\infty} (\mathcal{A} - H)$, and of translation length $2\ell$.
\end{lemma}

\begin{proof}
Let $\ell \in \NN^{*}$. Choose the $M$-parallel wall  $M'$  contained in $H$ such that $\dist_{\mathcal{A}}(M,M')= \ell$. Let $H'$ be the half-apartment of $\mathcal{A}$ with wall-boundary $M'$ and containing $\St(\sigma, \partial_{\infty} \mathcal{A})$ in its ideal boundary $\partial_{\infty} H'$.  By Lemma \ref{lem::existance_elliptic} there is $h \in G_{\st(\sigma,\partial_{\infty} \mathcal{A})}^{0}$ fixing $H'$ point-wise and moving away the half apartment $\mathcal{A} - H'$. Then $\mathcal{B}:=(\mathcal{A} - H') \cup h(\mathcal{A} - H')$ is again an apartment of $\Delta$ that contains $\mathcal{A} - H$ as a half-apartment.  Let $H'':=\mathcal{B} - (\mathcal{A} - H)$ and notice this half-apartment has $M$ as a boundary wall. Apply $k$ to $H''$, and so by our construction $k(H'') \cap \mathcal{A}= M$. Then $\mathcal{C}:= H \cup k(H'')$ is again an apartment in $\Delta$. We apply again Lemma \ref{lem::existance_elliptic} and find an element $h' \in G_{\st(\sigma,\partial_{\infty} \mathcal{A})}^{0}$ with $h'(H)=H$ point-wise and $h'k(H'') \subset \mathcal{A}$. By tracking the orientation during the construction, one notices $h'kh$ is a hyperbolic element of translation axis perpendicular to the wall $M$, containing $\mathcal{A} - H$ in its translation apartment, and with attracting end-point in $\partial_{\infty} (\mathcal{A} - H)$. The translation length of $h'kh$ is exactly $2 \ell$.
\end{proof}

\medskip
Let $G$ be a closed, strongly transitive and type-preserving subgroup of $\Aut(\Delta)$. Let $\mathcal{A}$ be an apartment of $\Delta$, an  ideal chamber $c$ of the spherical apartment $\partial_{\infty} \mathcal{A}$, and $x_0$ a special vertex of $\mathcal{A}$. Then we have the well known \textbf{polar decomposition}  $G=G_{x_0} AG_{x_0}$, where $A:= \{g \in \Stab_{G_0}(\mathcal{A}) \; \vert \; g \text{ hyperbolic}\} \cup \{e\}$ is the subgroup of euclidean translations in the apartment $\mathcal{A}$ (for a proof see for example \cite[Lemma 4.7]{Cio}). As well, by the well known \textbf{Bruhat  decomposition} we have $G= G_{c}WG_{c}=G_{c}^{0}\Stab_{G}(\mathcal{A})G_{c}^{0} $, where $W$ is the finite Weyl group associated with $\Delta$ and its elements can be lifted (not uniquely) to elements in $\Stab_{G}(\mathcal{A})$.

In what follows we will consider unitary representations $(\pi, \mathcal{H})$ of $G$ that are strongly continuous. The idea of the proof of Theorem \ref{thm::main} is the same as in \cite{BM00b}.


\begin{proof}[Proof of Theorem \ref{thm::main}]
Fix a unitary representation $(\pi, \mathcal{H})$ of $G$ without non-zero $G$-invariant vectors.  Suppose there exist two non-zero vectors $v_1,v_2 \in \mathcal{H}$ such that the $(v_1,v_2)$-matrix coefficient does not vanish at $\infty$. As $G=G_{x_0} A G_{x_0}$ there exist a sequence $\{a_n\}_{n\geq 1} \subset A$ and non-zero vectors $v_0,w_0 \in \mathcal{H}$ such that $a_n \xrightarrow[n \to \infty]{} \infty$ and $\vert \langle \pi(a_n)(v_0), w_0 \rangle \vert $ does not converges to $0$, when $n \to \infty$ (see for example \cite[Lemma 2.9]{Cio}).

As $a_n \xrightarrow[n \to \infty]{} \infty$ and $\Delta \cup \partial_{\infty} \Delta$ is a compact space with respect to the cone topology, by extracting a subsequence, we can suppose in addition $\{a_n(x_0)\}_{n \geq 1}$ converges to an ideal point $\xi \in \partial_{\infty} \mathcal{A}$ contained in the interior of a minimal ideal simplex  $\sigma \subset \partial_{\infty} \mathcal{A}$. Then apply Proposition \ref{prop:BM} to our sequence $\{a_n\}_{n \geq 1} \subset \Stab_{G}(\mathcal{A})$ of hyperbolic elements and the non-zero vectors $v_0,w_0 \in \mathcal{H}$ to obtain a non-zero vector $w \in \mathcal{H}$ such that $L:=\Stab_{G_{\st(\sigma,\partial_{\infty} \mathcal{A})}^{0}}(w)$ is of finite index in $G_{\st(\sigma,\partial_{\infty} \mathcal{A})}^{0}$.

Then consider the $(w,w)$-matrix coefficient on $G$, $f : G \to \CC$, given by $f(g):=  \langle \pi(g)(w), w \rangle$. This map $f$ is continuous with respect to the topology on $G$ and $L$-bi-invariant (i.e., $f(\ell_1g\ell_2)=f(g)$ for every $\ell_1, \ell_2 \in L$).

Choose some ideal chamber $c$ with $\sigma \subset c \in \Ch(\partial_{\infty} \mathcal{A})$. Then we have $G_{\st(\sigma,\partial_{\infty} \mathcal{A})}^{0} \leq G_c$. If $\sigma$ is just a sub-simplex of $c$, then $G_{\st(\sigma,\partial_{\infty} \mathcal{A})}^{0}$ is of infinite index in $G_c^{0}$, the latter being a normal subgroup of $G_c$. Still, one can use Bruhat decomposition $G= G_{c}WG_{c}=G_{c}^{0}\Stab_{G}(\mathcal{A})G_{c}^{0}$ and write $G=L H L$, where $H$ is a subset of $G$ containing many elements from  $\Stab_{G}(\mathcal{A})$ (e.g., hyperbolic elements), as well as many other elements.

By Lemma \ref{lem::sequence_elliptic_elements}, consider a sequence $\{k_n\}_{n \geq 1} \subset G_{x_0}$ and a sequence of half-apartments $\{H_n\}_{n \geq 1}$ in $\mathcal{A}$ with the required properties. By the continuity of the map $f$ we have $$f(L k_n L) \xrightarrow[n \to \infty]{} f(L)=1.$$

Let $j \in \NN$. By Lemma \ref{lem::hyper_elements}, applied to every $k_n$, the double coset $G_{\st(\sigma,\partial_{\infty} \mathcal{A})}^{0}k_nG_{\st(\sigma,\partial_{\infty} \mathcal{A})}^{0}$ contains a hyperbolic element $g_n=h_n k_n h_n'$, with $h_n,h_n' \in G_{\st(\sigma,\partial_{\infty} \mathcal{A})}^{0}$, of translation length $2j$, containing the half-apartment $H_n$ in its translation apartment, of translation axis perpendicular to boundary wall of $H_n$, and of attracting end-point in $\partial_{\infty} (H_n) \subset \partial_{\infty} \mathcal{A}$. Then $\{g_n = h_n k_n h_n'\}_{n \geq 1}$ admits a subsequence, by simplicity denoted by $\{g_n\}_{n \geq 1}$, that converges to an element $g_j \in A$ that is hyperbolic, of attracting endpoint in  $\partial_{\infty} (H_1) \subset \partial_{\infty} \mathcal{A}$, of translation length $2j$, and whose translation axis is perpendicular to the boundary wall of $H_1$.

Moreover, as $L$ has finite index in $G_{\st(\sigma,\partial_{\infty} \mathcal{A})}^{0}$ and $h_n,h_n' \in G_{\st(\sigma,\partial_{\infty} \mathcal{A})}^{0}$, by passing to a subsequence $\{g_{n_i}\}_{i \geq 1}$ we can write $h_{n_i} =\sigma \ell_{n_i}$, $h_{n_i}'=\ell_{n_i}' \sigma'$, with $\ell_{n_i},\ell_{n_i}' \in L$ and $\sigma, \sigma' \in G_{\st(\sigma,\partial_{\infty} \mathcal{A})}^{0}$. Then $$\lim\limits_{i \to \infty}\ell_{n_i} k_{n_i} \ell_{n_i}' = \sigma^{-1} g_j (\sigma')^{-1} $$
and because $g_j G_{\st(\sigma,\partial_{\infty} \mathcal{A})}^{0}g_j^{-1}= G_{\st(\sigma,\partial_{\infty} \mathcal{A})}^{0}$ (see Lemma \ref{lem::contr_group_alpha}) we can write further as 
$$\lim\limits_{i \to \infty}\ell_{n_i} k_{n_i} \ell_{n_i}' =  g_j \sigma_j, \text{ with } \sigma_j\in G_{\st(\sigma,\partial_{\infty} \mathcal{A})}^{0}.$$

Then as well by the continuity of the map $f$ we have 
$$f(L k_n L) \xrightarrow[n \to \infty]{} f(Lg_{j}\sigma_{j} L)$$
and therefore $f(Lg_{j}\sigma_{j} L)=1= \langle \pi(g_{j}\sigma_{j})(w), w \rangle$ for every $j \in \NN$. We have the equality $ \langle \pi(g_{j}\sigma_{j})(w), w \rangle=  \langle w, w \rangle$ if and only if $w$ is $g_{j}\sigma_{j}$-invariant, for every $j \in \NN$. Since $L$ has finite index in $G_{\st(\sigma,\partial_{\infty} \mathcal{A})}^{0}$, we can choose $\{\sigma_j\}_{j \in \NN} \subset  G_{\st(\sigma,\partial_{\infty} \mathcal{A})}^{0}/L$ that is a finite set.  As $g_j g_{j+i}^{-1}$ is still a hyperbolic element of $A$, for every $i,j \in \NN$, one can conclude there is a hyperbolic element $a \in A - \{e\}$ such that $w$ is $a$-invariant. Then by standard methods as Mautner's phenomenon (see for example  \cite[Prop. 2.12]{Cio}) applied to the sequence $\{a^{n}\}_{n \geq 1}$, resp., $\{a^{-n}\}_{n \geq 1}$, one has $w$ is $U_a^{-}$ and $U_a^{+}$-invariant.  If in addition $G$ is topologically simple, by \cite[Corollary 4.18]{Cio}, the non-zero vector $w$ is $G$-invariant. This is a contradiction with our assumption at the beginning of the proof, and so the theorem is proven.
\end{proof}

\end{document}